\documentclass[12pt, twoside, leqno]{article}

\usepackage{amsmath}
\usepackage{amsthm}
\usepackage{amssymb}
\usepackage{enumerate}

\newtheorem{theorem}{Theorem}[section]
\newtheorem{lemma}[theorem]{Lemma}
\newtheorem{proposition}[theorem]{Proposition}
\newtheorem{corollary}[theorem]{Corollary}

\newtheorem*{maintheorem}{Theorem \ref{lipschitz}}

\theoremstyle{definition}
\newtheorem{definition}[theorem]{Definition}

\begin{document}

\baselineskip=17pt

\title{On coarse Lipschitz embeddability\\ into $c_0(\kappa)$}
\author{Andrew Swift \\
E-mail:  ats0@math.tamu.edu  }
\date{}
\maketitle

\renewcommand{\thefootnote}{}



\renewcommand{\thefootnote}{\arabic{footnote}}
\setcounter{footnote}{0}

\begin{abstract}
In 1994, Jan Pelant proved that a metric property related to the notion of paracompactness called the uniform Stone property characterizes a metric space's uniform embeddability into $c_0(\kappa)$ for some cardinality $\kappa$.  In this paper it is shown that coarse Lipschitz embeddability of a metric space into $c_0(\kappa)$ can be characterized in a similar manner.  It is also shown that coarse, uniform, and bi-Lipschitz embeddability into $c_0(\kappa)$ are equivalent notions for normed linear spaces.
\end{abstract}

\section{Introduction}
Aharoni showed in 1974 \cite{aharoni} that for any $K>6$, every separable metric space $K$-Lipschitz embeds into $c_0^+$ (where the positive cone of $c_0$, denoted $c_0^+$, is the set $\{(x_i)_{i=1}^\infty\in c_0\ |\ x_i\geq 0 \mbox{ for all }i\in \mathbb{N}\}$ with metric inherited from $c_0$); and also that $\ell_1$ does not $K$-Lipschitz embed into $c_0$ for any $K<2$.  In 1978, Assouad \cite{assouad} improved Aharoni's result and showed that for any $K>3$, every separable metric space $K$-Lipschitz embeds into $c_0^+$.  The final improvement for $c_0^+$ came when Pelant showed in 1994 \cite{pelant2} that every separable metric space $3$-Lipschitz embeds into $c_0^+$ and that $\ell_1$ cannot be $K$-Lipschitz embedded into $c_0^+$ for any $K<3$.  This left open the problem of finding the best constant for bi-Lipschitz embedding a separable metric space into $c_0$ until Kalton and Lancien showed in 2008 \cite{kalton2} that every separable metric space $2$-Lipschitz embeds into $c_0$.  They do this by showing that every separable metric space has property $\Pi(2)$, property $\Pi(\lambda)$ being a sufficient criterion they define for implying $\lambda$-Lipschitz embeddability into $c_0$ for a separable metric space.  Recently, Baudier and Deville \cite{baudier} have made a slight improvement to Kalton and Lancien's proof using a related criterion $\pi({\lambda})$ to show that every separable metric space $2$-Lipschitz embeds into $c_0$ via a special kind of bi-Lipschitz embedding.

It is natural to ask whether a similar result holds for non-separable metric spaces.  In particular, does every metric space bi-Lipschitz embed into $c_0(\kappa)$ for large enough cardinality $\kappa$?  The answer to this question comes from the theory of uniform spaces.  In 1948, Stone \cite{stone1} showed that every metric space is paracompact.  In 1960 \cite{stone2}, he asked whether every uniform cover of a metric space has a locally finite uniform refinement (or equivalently a point-finite uniform refinement).  That is, does every metric space possess a uniform analog of paracompactness (a property which has come to be called the uniform Stone property)?  The question was answered in the negative by Pelant \cite{pelant1} and Shchepin \cite{shchepin}, who showed that $\ell_\infty(\Gamma)$ fails to have the uniform Stone property if $\Gamma$ has large enough cardinality.  Moreover, Pelant (see \cite{pelant2}) has shown that the uniform Stone property characterizes uniform embeddability into $c_0(\kappa)$ for some $\kappa$ and thus $\ell_\infty (\Gamma)$ does not even uniformly embed into any $c_0(\kappa)$ when $\Gamma$ has large enough cardinality.

It remains an open problem in the nonlinear theory of Banach spaces to determine whether a Banach space's uniform embeddability into a given Banach space Y is equivalent to its coarse embeddability into Y, and so one is led to ask whether a characterization of coarse embeddability into $c_0(\kappa)$ involving covers also exists.
We suggest a natural candidate for such a ``coarse Stone property", and show this to be at least a necessary condition for coarse embeddability into $c_0(\kappa)$.
Related to this property, however, is a natural modulus $\Delta_X^{(c)}$ which can be defined for any metric space $X$ and whose growth can be used to characterize coarse Lipschitz embeddability (and also bi-Lipschitz embeddability) into $c_0(\kappa)$.  The main result is the following theorem.

\begin{maintheorem}
Let $(X,d_X)$ be a metric space with density character $\kappa$.
If there are $1\leq C<\infty$ and $D<\infty$ such that $\Delta_X^{(c)}(R)\leq CR+D$ for all $R<\infty$; then for any $\lambda>0$, any $K> 2(C+\lambda)$, and any $L>\frac{(C+\lambda)D}{\lambda}$; there exists a coarse Lipschitz embedding $f\colon X\to c_0^+(\kappa)$ such that \[d_X(x,y)-L\leq \|f(x)-f(y)\|_{\infty}\leq Kd_X(x,y)\] for every $x,y\in X$.  If $D=0$, then it is possible to take $L=0$.
\end{maintheorem}

\section{Preliminaries and notation}
Let $(X,d_X)$ and $(Y,d_Y)$ be metric spaces.
Given $x\in X$ and $r\in [0,\infty)$, we will denote by $B_r(x)$ the open ball of radius $r$ centered at $x$.
For a map $f\colon X\to Y$, the \emph{modulus of continuity} (or \emph{modulus of expansion}) of $f$ is the function $\omega_f\colon [0,\infty )\to [0,\infty]$ defined by \[\omega_f(t)=\sup \{d_Y(f(x_1),f(x_2))\ |\ d_X(x_1,x_2)\leq t\},\]
and the \emph{modulus of compression} of $f$ is the function $\rho_f\colon [0,\infty )\to [0,\infty]$ defined by \[\rho_f(t)=\inf \{d_Y(f(x_1),f(x_2))\ |\ d_X(x_1,x_2)\geq t\}.\]
Note that $\omega_f$ and $\rho_f$ are non-decreasing and for all $x_1,x_2\in X$,
\[\rho_f(d_X(x_1,x_2))
\leq d_Y(f(x_1),f(x_2))
\leq \omega_f(d_X(x_1,x_2)).\]
A map $f$ is said to be \emph{uniformly continuous} (or simply \emph{uniform}) if $\lim\limits_{t\to 0}\omega_f(t)=0$ and is called a \emph{uniform embedding} if furthermore $\rho_f(t)>0$ for all $t>0$.
A map $f$ is said to be \emph{coarse} (or sometimes \emph{coarsely continuous}) if $\omega_f(t)<\infty$ for all $t\in [0,\infty)$ and is called a \emph{coarse embedding} if furthermore $\lim\limits_{t\to \infty}\rho_f(t)=\infty$.
A map $f$ is called a \emph{coarse Lipschitz embedding} (or a \emph{quasi-isometric embedding}, especially in the literature of geometric group theory) if there exist $A\geq 1$ and $B\geq 0$ such that $\omega_f(t)\leq At+B$ and $\rho_f(t)\geq \frac{1}{A}t-B$ for all $t$ and is called a \emph{bi-Lipschitz embedding} if furthermore $B$ can be taken to be equal to $0$.
The \emph{Lipschitz constant} of $f$ is defined to be
\[\mathrm{Lip}(f)=\sup\left\{\frac{d_Y(f(x_1),f(x_2))}{d_X(x_1,x_2)}\ |\ x_1\neq x_2\right\}.\]
A map $f$ is said to be \emph{Lipschitz} if $\mathrm{Lip}(f)<\infty$.
If $f$ is injective, the \emph{distortion} of $f$ is defined to be $\mathrm{dist}(f)=\mathrm{Lip}(f)\cdot \mathrm{Lip}(f^{-1})$.
If $\mathrm{dist}(f)\leq K$, then $f$ is called a $K$-Lipschitz embedding.

Given $a,b\in \mathbb{R}^+$, $S\subseteq X$ is called \emph{$a$-separated} if $d_X(s_1,s_2)\geq a$ for all $s_1,s_2\in S$, is called \emph{$b$-dense} in $X$ if $d_X(x,S)\leq b$ for all $x\in X$, and is called an \emph{$(a,b)$-skeleton} of $X$ if it is $a$-separated and $b$-dense in $X$.  Given a skeleton $S$ of $X$, there is a coarse embedding $f\colon X\to S$ such that $\sup \{d_X(f(x),S)\}_{x\in X}<\infty$ (just map every point of the space to a nearest point in the skeleton), and so questions about coarse embeddings of metric spaces can be reduced to questions about coarse embeddings of uniformly discrete metric spaces.  By Zorn's Lemma, every $a$-separated set can be extended to a maximal (in the sense of set containment) $(a,a)$-skeleton of $X$.  Note that $|S|\leq \mathrm{dens}(X)$ (where $|S|$ denotes the cardinality of $S$ and where $\mathrm{dens}(X)$, the \emph{density character} of $X$, is the smallest cardinality of a set dense in $X$) for any skeleton $S$ of $X$.  And if $X$ is a normed linear space, then $X=\overline{\mathrm{span}(S)}$ (the closed linear span of $S$) for any skeleton of $X$ (or else $S$ is not $b$-dense in $X$ for any $b\in \mathbb{R}^+$), and so in this case $|S|=\mathrm{dens}(X)$.  The following lemma holds.

\begin{lemma}
\label{prelim}
Let $(X, \|\cdot\|_X)$ be a normed linear space and $(Y,d_Y)$ a metric space.  If there exists a map $f\colon X\to Y$ such that $\lim_{t\to\infty}\rho_f(t)=\infty$, then $\mathrm{dens}(X)\leq \mathrm{dens}(Y)$.
\end{lemma}

\begin{proof}
Let $a>0$ be such that $\rho_f(a)>0$, and let $S$ be an $(a,a)$-skeleton of $X$.
Then $f|_S$ is injective and maps $S$ to a $\rho_f(a)$-separated subset of $Y$.
And so  
\[\mathrm{dens}(X)= |S|=|f(S)|\leq \mathrm{dens}(Y).\qedhere\]
\end{proof}

\

A family of sets $\mathcal{U}\subseteq \mathcal{P}(X)$ (where $\mathcal{P}(X)$ denotes the power set of $X$) is called a \emph{cover} of $X$ if $\bigcup\limits_{U\in \mathcal{U}}U=X$.  
The \emph{diameter} of a cover $\mathcal{U}$ of $X$ is \[\mathrm{diam}(\mathcal{U})=\sup\{\mathrm{diam}(U)\ |\ U\in \mathcal{U}\}\]
where for $U\subseteq X$, $\mathrm{diam}(U)=\sup \{d_X(x_1,x_2)\ |\ x_1,x_2\in U\}$ is the \emph{diameter} of $U$.  
The \emph{Lebesgue number} of a cover $\mathcal{U}$ of $X$ is 
\begin{align*}
 \mathcal{L}(\mathcal{U})=\sup \{d\in [0,\infty)\ |\ & \mbox{For every }E\subseteq X\mbox{ such that } \mathrm{diam}(E)<d,\\
 & \mbox{there is }U\in \mathcal{U}\mbox{ such that } E\subseteq U\}.
\end{align*}
Note that by definition $\mathcal{L}(\{X\})=\infty$ and $\mathcal{L}(\{\{x\}\}_{x\in X})=\inf_{x\neq y}d_X(x,y)$.  
A cover $\mathcal{U}$ of $X$ is called a \emph{uniform} cover if $\mathcal{L}(\mathcal{U})>0$ and is called a \emph{uniformly bounded} (or \emph{coarse}) cover if $\mathrm{diam}(\mathcal{U})<\infty$.  
A cover $\mathcal{U}$ of $X$ is called \emph{point-finite} if for all $x\in X$, there are only finitely many $U\in \mathcal{U}$ such that $x\in U$.  
A cover $\mathcal{V}$ of $X$ is called a \emph{refinement} of the cover $\mathcal{U}$ of $X$; and in this case, $\mathcal{V}$ is said to \emph{refine} $\mathcal{U}$; if for all $V\in \mathcal{V}$, there is $U\in \mathcal{U}$ such that $V\subseteq U$.
We have the following lemma.

\begin{lemma}
\label{covercardinality}
Let $(X,d_X)$ be a metric space with density character $\kappa$, and let $\mathcal{U}$ be a point-finite uniform cover of $X$.  There exists $\mathcal{V}\subseteq\mathcal{U}$ such that $|\mathcal{V}|\leq \kappa$ and such that $\mathcal{V}$ is a point-finite cover of $X$ with $\mathcal{L}(\mathcal{V})= \mathcal{L}(\mathcal{U})$.
\end{lemma}

\begin{proof}
Let $\{x_\tau\}_{\tau<\kappa}$ be a dense set in $X$ and let 
\[\mathcal{V}=\{U\in \mathcal{U}\ |\ x_\tau\in U \mbox{ for some } \tau<\kappa\}.\]
Then $|\mathcal{V}|\leq \kappa$ since $\mathcal{U}$ is point-finite.
Now take any $A\subseteq X$ such that $\mathrm{diam}(A)<\mathcal{L}(\mathcal{U})$.
If $A=\emptyset$, then clearly there is $V\in \mathcal{V}$ such that $A\subseteq V$, so suppose $A\neq \emptyset$ and let $x\in A$.
Choose any $0<r<\mathcal{L}(\mathcal{U})-\mathrm{diam}(A)$ and let $B=A\cup B_r(x)$.
Then $\mathrm{diam}(B)<\mathcal{L}(\mathcal{U})$, and so there is $U\in \mathcal{U}$ such that $B\subseteq U$.
But there is $\tau<\kappa$ such that $x_\tau \in B_r(x)\subseteq B\subseteq U$ by the density of $\{x_\tau\}_{\tau<\kappa}$, and so $U\in \mathcal{V}$.
Therefore $\mathcal{V}$ is a cover of $X$ such that $\mathcal{L}(\mathcal{V})\geq \mathcal{L}(\mathcal{U})$.
Furthermore, $\mathcal{V}$ is point-finite and $\mathcal{L}(\mathcal{V})=\mathcal{L}(\mathcal{U})$ because $\mathcal{V}\subseteq \mathcal{U}$.
\end{proof}
\section{Characterizing embeddability}

We start by defining the uniform Stone property, which characterizes a metric space's uniform embeddability into some $c_0(\kappa)$.  One can view the property as a generalization of having finite (uniform) covering dimension, which is the natural notion of dimension associated with the class of uniform spaces.

\begin{definition}
\label{usp}
A metric space $(X,d_X)$ is said to have the \emph{uniform Stone property} if every uniform cover of $X$ has a point-finite uniform refinement.
\end{definition}

The class of coarse spaces has a similar notion of dimension associated with it, called asymptotic dimension.  It has become clear in recent years that many ideas in the uniform theory have useful analogues in the coarse theory, and so the motivation behind the following definition is to generalize the property of having finite asymptotic dimension in a manner similar to the way the uniform Stone property generalizes having finite covering dimension.

\begin{definition}
\label{csp}
A metric space $(X,d_X)$ is said to have the \emph{coarse Stone property} if every uniformly bounded cover of $X$ refines a point-finite uniformly bounded cover.
\end{definition}

We immediately turn to more quantitative formulations.  
Given a metric space $(X,d_X)$, define the functions $\Delta_X^{(u)},\Delta_X^{(c)}\colon$ $[0,\infty )\to [0,\infty]$ by 
\[\Delta_X^{(u)}(r)  = \sup \{\mathcal{L}(\mathcal{U})\ |\ \mathcal{U}\mbox{ is a point-finite cover of } X \mbox{ and } \mathrm{diam}(\mathcal{U})\leq r\}\]
and \[\Delta_X^{(c)}(R)  =  \inf \{\mathrm{diam}(\mathcal{U})\  |\ \mathcal{U}\mbox{ is a point-finite cover of } X \mbox{ and } \mathcal{L}(\mathcal{U})\geq R\}.\]

\begin{proposition}
Let $(X,d_X)$ be a metric space.
\begin{enumerate}[(i)]
\item $X$ has the uniform Stone property if and only if $\Delta_X^{(u)}(r)>0$ for all $r>0$.
\item $X$ has the coarse Stone Property if and only if $\Delta_X^{(c)}(R)<\infty$ for all $R\in [0,\infty)$.  
\end{enumerate}
\end{proposition}

\begin{proof}
(i):  Suppose first that $X$ has the uniform Stone property, and take any $r>0$.
 Let $\mathcal{U}=\{B_{r/2}(x)\}_{x\in X}$, and note that $\mathcal{U}$ is a uniform cover of $X$ with $\mathrm{diam}(\mathcal{U})\leq r$.
By assumption, $\mathcal{U}$ has a point-finite uniform refinement $\mathcal{V}$, and so $\Delta_X^{(u)}(r)\geq \mathcal{L}(\mathcal{V})>0$.  Conversely, suppose $\Delta_X^{(u)}(r)>0$ for all $r>0$, and take any uniform cover $\mathcal{U}$ of $X$.
Since $\mathcal{U}$ is uniform, there is $r>0$ such that $\mathcal{L}(\mathcal{U})>r$.
And by assumption, there is a point-finite cover $\mathcal{V}$ of $X$ such that $0<\mathcal{L}(\mathcal{V})\leq \Delta_X^{(u)}(r)$ and $\mathrm{diam}(\mathcal{V})\leq r$.  But then $\mathcal{V}$ is a point-finite uniform refinement of $\mathcal{U}$, and so $X$ has the uniform Stone property.\newline
(ii):  Suppose first that $X$ has the coarse Stone property, and take any $R\in [0,\infty)$.
 Let $\mathcal{U}=\{B_{R}(x)\}_{x\in X}$, and note that $\mathcal{U}$ is a uniformly bounded cover of $X$ with $\mathcal{L}(\mathcal{U})\geq R$.
By assumption, $\mathcal{U}$ refines a point-finite uniformly bounded cover $\mathcal{V}$, and so $\Delta_X^{(c)}(R)\leq \mathrm{diam}(\mathcal{V})<\infty$.  Conversely, suppose $\Delta_X^{(c)}(R)<\infty$ for all $R\in[0,\infty)$, and take any uniformly bounded cover $\mathcal{U}$ of $X$.
Since $\mathcal{U}$ is uniformly bounded, there is $R\in [0,\infty)$ such that $\mathrm{diam}(\mathcal{U})<R$.
And by assumption, there is a point-finite cover $\mathcal{V}$ of $X$ such that $\Delta_X^{(c)}(R)\leq\mathrm{diam}(\mathcal{V})<\infty $ and $\mathcal{L}(\mathcal{V})\geq R$.  But then $\mathcal{V}$ is a point-finite uniformly bounded cover refined by $\mathcal{U}$, and so $X$ has the coarse Stone property.
\end{proof}

From this point forward, whenever we write ``uniform Stone property'' or ``coarse Stone property'', we are using the equivalent formulations of these properties in terms of $\Delta_X^{(u)}$ and $\Delta_X^{(c)}$, respectively.  We have the following lemma.

\begin{lemma}
\label{easy}
Let $(X,d_X)$ be a metric space and let $r,R\in (0,\infty)$.
\begin{enumerate}[(i)]
\item $\Delta_X^{(c)}$, $\Delta_X^{(u)}$ are non-decreasing functions.
\item If $\Delta_X^{(c)}(R)<\infty$, then $\Delta_X^{(u)}(\Delta_X^{(c)}(R)+\varepsilon)\geq R$ for all $\varepsilon>0$.
\item  If $\Delta_X^{(u)}(r)>0$, then $\Delta_X^{(c)}(\Delta_X^{(u)}(r)-\varepsilon) \leq r$ for all $0<\varepsilon<\Delta_X^{(u)}(r)$.
\item $X$ has the uniform Stone property if and only if $\lim\limits_{R\to 0}\Delta_X^{(c)}(R)=0$.
\item $X$ has the coarse Stone property if and only if $\lim\limits_{r\to \infty}\Delta_X^{(u)}(r)=\infty$.

\end{enumerate}
\end{lemma}

\begin{proof}
(i):  This is clear from the definitions.\newline
(ii):  If $\Delta_X^{(c)}(R)<\infty$, then there is a point-finite cover $\mathcal{U}$ of $X$ such that $\mathrm{diam}(\mathcal{U})\leq\Delta_X^{(c)}(R)+\varepsilon$ and $\mathcal{L}(\mathcal{U})\geq R$.  Thus $\Delta_X^{(u)}(\Delta_X^{(c)}(R)+\varepsilon)\geq \mathcal{L}(\mathcal{U})\geq R$.\newline
(iii):  If $\Delta_X^{(u)}(r)>0$, then there is a point-finite cover $\mathcal{U}$ of $X$ such that $\mathcal{L}(\mathcal{U})\geq \Delta_X^{(u)}(r)-\varepsilon$ and $\mathrm{diam}(\mathcal{U})\leq r$.  Thus $\Delta_X^{(c)}(\Delta_X^{(u)}(r)-\varepsilon) \leq \mathrm{diam}(\mathcal{U})\leq r$.\newline
(iv):  Suppose first that $X$ has the uniform Stone property and take any $\varepsilon>0$.  Then by assumption, $\Delta_X^{(u)}(\varepsilon)>0$, and so $\Delta_X^{(c)}(R)\leq \varepsilon$ for any $R<\Delta_X^{(u)}(\varepsilon)$ by parts (i) and (iii).  Thus $\lim_{R\to 0}\Delta_X^{(c)}(R)=0$.
Conversely, suppose $\lim_{R\to 0}\Delta_X^{(c)}(R)=0$ and take any $r>0$.
Let $R>0$ be such that $\Delta_X^{(c)}(R)<r$.  Then $\Delta_X^{(u)}(r)\geq R>0$ by part (ii), and so $X$ has the uniform Stone property.\newline
(v):  Suppose first that $X$ has the coarse Stone property and take any $N\in \mathbb{N}$.  Then by assumption, $\Delta_X^{(c)}(N)<\infty$, and so $\Delta_X^{(u)}(r)\geq N$ for any $r>\Delta_X^{(c)}(N)$ by parts (i) and (ii).  Thus $\lim_{r\to \infty}\Delta_X^{(u)}(r)=\infty$.
Conversely, suppose $\lim_{r\to \infty}\Delta_X^{(u)}(r)=\infty$ and take any $R\in [0,\infty)$.
Let $r\in [0,\infty)$ be such that $\Delta_X^{(u)}(r)>R$.  Then $\Delta_X^{(c)}(R)\leq r<\infty$ by part (iii), and so $X$ has the coarse Stone property.
\end{proof}

It is clear that if $X$ is a metric space with finite diameter, then there are $C<1$ and $D\in[0,\infty)$ such that $\Delta_X^{(c)}(R)\leq CD+R$ for all $R\in[0,\infty)$ (indeed, $\Delta_X^{(c)}(R)\leq \mathrm{diam}(X)$ for all $R\in[0,\infty)$ in this case).  The converse is also true.

\begin{lemma}
\label{C<1}
Let $(X,d_X)$ be a metric space.
If $C<1$ and $D\in[0,\infty)$ are such that $\Delta_X^{(c)}(R)\leq CR+D$ for all $R\in[0,\infty)$, then $\mathrm{diam}(X)\leq\frac{D}{1-C}$.
\end{lemma}

\begin{proof}  
Take any $0<\varepsilon< 1-C$.  
Then for any $R>\frac{D}{1-C-\varepsilon}$, 
\[
\Delta_X^{(c)}(R)  \leq  CR+D  <  CR+(1-C-\varepsilon)R  =  (1-\varepsilon)R.
\]
So suppose there exist $x,y\in X$ such that $d_X(x,y)>\frac{D}{1-C-\varepsilon}$.
Then
\[
\Delta_X^{(c)}((1+\varepsilon)d_X(x,y))  <  (1-\varepsilon)((1+\varepsilon)d_X(x,y)) =  (1-\varepsilon^2)d_X(x,y).
\] 
Thus there is a point-finite cover $\mathcal{U}$ of $X$ with $\mathcal{L}(\mathcal{U})\geq (1+\varepsilon)d(x,y)>d(x,y)$ satisfying $\mathrm{diam}(\mathcal{U})<(1-\varepsilon^2)d(x,y)<d(x,y)$.
But this is a contradiction since $\mathrm{diam}(\{x,y\})=d(x,y)$.
Therefore $d(x,y)\leq \frac{D}{1-C-\varepsilon}$ for every $x,y\in X$.
Since $0<\varepsilon<1-C$ was arbitrary, $\mathrm{diam}(X)\leq\frac{D}{1-C}$.
\end{proof}

In some cases it might be more natural to find bounds for $\Delta_X^{(u)}$ rather than $\Delta_X^{(c)}$ or vice versa.  Lemma \ref{easy} provides a way of switching from one to the other and this is especially easy in the case below.

\begin{lemma}
\label{lineartype}
Let $(X,d_X)$ be a metric space.  Given $C\in(0,\infty)$, $\Delta_X^{(c)}(R)\leq CR$ for all $R\in[0,\infty)$ iff $\Delta_X^{(u)}(r)\geq \frac{1}{C}r$ for all $r>0$.
\end{lemma}

\begin{proof}
Suppose first that $\Delta_X^{(c)}(R)\leq CR$ for all $R\in[0,\infty)$.
Take any $r>0$ and $0<\varepsilon<r$.
Then $\Delta_X^{(c)}\left(\frac{r-\varepsilon}{C}\right)\leq r-\varepsilon$ and so by Lemma \ref{easy},
 \[\Delta_X^{(u)}(r)\geq \Delta_X^{(u)}\left(\Delta_X^{(c)}\left(\frac{r-\varepsilon}{C}\right)+\varepsilon\right)\geq \frac{1}{C}\cdot (r-\varepsilon ).\]
Since $0<\varepsilon<r$ was arbitrary, $\Delta_X^{(u)}(r)\geq \frac{1}{C}r$ for all $r>0$.\newline
Now suppose $\Delta_X^{(u)}(r)\geq \frac{1}{C}r$ for all $r>0$.
Take any $R\in[0,\infty)$ and $\varepsilon>0$.
Then $\Delta_X^{(u)}(C(R+\varepsilon))\geq R+\varepsilon$ and so by Lemma \ref{easy}, \[\Delta_X^{(c)}(R)\leq \Delta_X^{(c)}(\Delta_X^{(u)}(C(R+\varepsilon))-\varepsilon)\leq C(R+\varepsilon).\]  
Since $\varepsilon>0$ was arbitrary, $\Delta_X^{(c)}(R)\leq CR$ for all $R\in [0,\infty)$.
\end{proof}

\begin{lemma}
\label{norm}
Let $(X,\|\cdot\|_X)$ be a normed linear space.  The following are equivalent:  
\begin{enumerate}[(i)]
\item $\Delta_X^{(c)}(R)<\infty$ for some $R\in (0,\infty)$.
\item There is $C\in[0,\infty)$ such that $\Delta_X^{(c)}(R)\leq CR$ for all $R\in[0,\infty)$.
\item $X$ has the coarse Stone property.
\item $X$ has the uniform Stone property.
\end{enumerate}
\end{lemma}

\begin{proof}  
(i) $\Rightarrow$ (ii):  The implication follows by simply scaling any uniformly bounded point-finite cover $\mathcal{U}$ of $X$ such that $\mathcal{L}(\mathcal{U})\geq R$ (and in this case one can take $C=\frac{\Delta_X^{(c)}(R)}{R}$).  \newline
(ii) $\Rightarrow$ (iii):  Clear. \newline
(iii) $\Rightarrow$ (iv):  If $X$ has the coarse Stone property, then in particular, $\Delta_X^{(c)}(1)<\infty$.  Thus, $X$ has the uniform Stone property by (i) $\Rightarrow$ (ii) and Lemma \ref{lineartype}. \newline
(iv) $\Rightarrow$ (i):  Lemma \ref{easy}.
\end{proof}

The following two propositions show that the uniform and coarse Stone properties are hereditary in the sense that a uniformly/coarsely embedded subset of a metric space with the uniform/coarse Stone property has the uniform/coarse Stone property respectively.

\begin{proposition}
\label{embed2}
Let $(X, d_X)$ be a metric space and $(Y, d_Y)$ a metric space with the uniform Stone property.  
If there exists a uniform embedding $f\colon X\to Y$, then $X$ has the uniform Stone property.  
If $f$ is a bi-Lipschitz embedding and if there is $c>0$ such that $\Delta_Y^{(u)}(r)\geq cr$ for all $r>0$ , then $\Delta_X^{(u)}(r)\geq \frac{c}{\mathrm{dist}(f)} r$ for all $r>0$.
\end{proposition}

\begin{proof}
Fix $r>0$.
Since $f$ is a uniform embedding, $\lim_{t\to 0}\omega_f(t)=0$ and $\rho_f(t)>0$ for all $t>0$.
Take any $0<\varepsilon_1 <\rho_f(r)$ and $0<\varepsilon_2<\Delta_Y^{(u)}(\rho_f(r)-\varepsilon_1)$ and let $\mathcal{V}$ be a point-finite cover of $Y$ such that $\mathrm{diam}(\mathcal{V})\leq \rho_f(r)-\varepsilon_1$ and $\mathcal{L}(\mathcal{V})\geq \Delta_Y^{(u)}(\rho_f(r)-\varepsilon_1)-\varepsilon_2$.  
Let $\mathcal{U}=\{f^{-1}(V)\}_{V\in \mathcal{V}}$.
Then $\mathcal{U}$ is a cover of $X$.
Note that $\mathcal{U}$ inherits point-finiteness from $\mathcal{V}$.
And for any $V\in \mathcal{V}$, $\mathrm{diam}(f^{-1}(V))\leq r$ since $\mathrm{diam}(V)\leq \rho_f(r)-\varepsilon_1$.
This means $\mathrm{diam}(\mathcal{U})\leq r$.  
Thus,
\[\Delta_X^{(u)}(r)\geq \mathcal{L}(\mathcal{U})\geq \inf \omega_f^{-1}([\mathcal{L}(\mathcal{V}),\infty])\geq \inf \omega_f^{-1}([\Delta_Y^{(u)}(\rho_f(r)-\varepsilon_1)-\varepsilon_2,\infty])>0\]
by definition of $\Delta_X^{(u)}$, the assumptions on $\rho_f$ and $\omega_f$, and since $Y$ has the uniform Stone property.  
Thus, $X$ has the uniform Stone property.
The special case follows by bounding $\omega_f$, $\rho_f$, and $\Delta_Y^{(u)}$ with linear functions and letting $\varepsilon_1,\varepsilon_2\to 0$.
\end{proof}

\begin{proposition}
\label{embed1}
Let $(X, d_X)$ be a metric space and $(Y, d_Y)$ a metric space with the coarse Stone property.  If there exists a coarse embedding $f: X\to Y$, then $X$ has the coarse Stone property.  
If $f$ is a coarse Lipschitz embedding and there are $C,D\in[0,\infty)$ such that $\Delta_Y^{(c)}(R)\leq CR+D$ for all $R\in[0,\infty)$, then there are $C',D'\in[0,\infty)$ such that $\Delta_X^{(c)}(R)<C'R+D'$ for all $R\in[0,\infty)$.
If, in particular, $f$ is a bi-Lipschitz embedding and $D=0$, then $\Delta_X^{(c)}(R)\leq C\mathrm{dist}(f)R$ for all $R\in[0,\infty)$.
\end{proposition}

\begin{proof}
Fix $R\in[0,\infty)$.
Since $f$ is a coarse embedding, $\omega_f(t)<\infty$ for all $t\in[0,\infty)$ and $\lim_{t\to \infty} \rho_f(t)=\infty$.
Take any $\varepsilon>0$ and let $\mathcal{V}$ be a point-finite cover of $Y$ such that $\mathcal{L}(\mathcal{V})\geq \omega_f(R)+\varepsilon$ and $\mathrm{diam}(\mathcal{V})\leq \Delta_Y^{(c)}(\omega_f(R)+\varepsilon)+\varepsilon$.  
Let $\mathcal{U}=\{f^{-1}(V)\}_{V\in \mathcal{V}}$.
Then $\mathcal{U}$ is a cover of $X$.
Note that $\mathcal{U}$ inherits point-finiteness from $\mathcal{V}$.
Now take any $A\subseteq X$ such that $\mathrm{diam}(A)< R$.
Then $\mathrm{diam}(f(A))< \omega_f(R)+\varepsilon$, and so $f(A)\subseteq V$ for some $V\in \mathcal{V}$.
Therefore $A\subseteq f^{-1}(V)=U$ for some $U\in \mathcal{U}$.
Since $A\subseteq X$ was arbitrary, this means $\mathcal{L}(\mathcal{U})\geq R$.
Thus,
\begin{align*}
 \Delta_X^{(c)}(R)\leq \mathrm{diam}(\mathcal{U})
 &\le \sup \rho_f^{-1}([0,\mathrm{diam}(\mathcal{V})])\\
 &\le \sup \rho_f^{-1}([0,\Delta_Y^{(c)}(\omega_f(R)+\varepsilon)+\varepsilon])<\infty
\end{align*}
by definition of $\Delta_X^{(c)}$, the assumptions on $\rho_f$ and $\omega_f$, and since $Y$ has the coarse Stone property.  
Thus, $X$ has the coarse Stone property.
The special cases follow by bounding $\omega_f$, $\rho_f$, and $\Delta_Y^{(c)}$ with affine or linear functions and letting $\varepsilon\to 0$.
\end{proof}

\begin{proposition}
\label{c0}
For any cardinality $\kappa$, $\Delta_{c_0^+(\kappa)}^{(c)}(R)= R$ for all $R\in[0,\infty)$.
\end{proposition}

\begin{proof}
Take any $n\in \mathbb{N}$ and $R\in[0,\infty)$.
Given a finite subset $M$ of $\kappa$, denote the set $\{x\in (\mathbb{N}\cup \{0\})^\kappa\ |\ x_\xi=0 \mbox{ if }\xi\notin M\}$ by $\mathbb{N}^M$.
For each finite subset $M$ of $\kappa$ and $x\in \mathbb{N}^{M}$, let
\[
U_{M,x}=\left\{f\in c_0^+(\kappa)\ |\ f(\xi)\in \frac{x_\xi}{n}+\left[0, 2R+\frac{1}{n}\right)\mbox{ for all } \xi\in\kappa\right\}.
\]
Then for a fixed finite subset $M$ of $\kappa$ and a fixed $f\in c_0^+(\kappa)$, there are at most $(2n\lceil R\rceil+1)^{|M|}$ many $x\in \mathbb{N}^{M}$ such that $f\in U_{M,x}$.
Let \[\mathcal{U}=\left\{U_{M,x}\ |\ M \mbox{ is a finite subset of } \kappa \mbox{ and }x\in\mathbb{N}^{M}\right\}.\]
Now take any $f\in c_0^+(\kappa)$.
There is a finite subset $M$ of $\kappa$ such that $f(\xi)<\frac{1}{n}$ if $\xi \not\in M$, and in this case there is $x\in \mathbb{N}^{M}$ such that $B_R(f)\subseteq U_{M,x}$ (simply choose $x_\xi=\lfloor n(f(\xi)-R)\rfloor$ when $\xi\in M$ and $x_\xi=0$ otherwise).
Now suppose $M'\supsetneq M$ and  $x'\in \mathbb{N}^{M'}$ is such that $f\in U_{M',x'}$.
Then for all $\xi\in M'\backslash M$, $x_\xi'=0$ (or else $f(\xi)\geq\frac{1}{n}$, contradicting the choice of $M$).
Thus $U_{M',x'}=U_{M,y}$ for some $y\in \mathbb{N}^{M}$.  
This means that for every $f\in c_0^+(\kappa)$, $f\in U$ for only finitely many $U\in \mathcal{U}$.
By the above, $\mathcal{U}$ is a point-finite cover of $c_0^+(\kappa )$ refined by $\{B_R(f)\}_{f\in c_0^+(\kappa )}$ such that $\mathrm{diam}(\mathcal{U})=2R+\frac{1}{n}$.
Since every $A\subseteq c_0^+(\kappa )$ such that $\mathrm{diam}(A)< 2R$ is contained in a ball of radius $R$ (centered at $\left(\frac{\sup \pi_\tau(A)-\inf \pi_\tau(A)}{2}\right)_{\tau<\kappa}$, where $\pi_\tau$ is the $\tau$-th coordinate functional), this means $\mathcal{L}(\mathcal{U})\geq 2R$.
Thus, since $n\in \mathbb{N}$ was arbitrary, $\Delta_{c_0^+(\kappa )}^{(c)}(2R)\leq 2R$.
And so, by Lemma \ref{C<1}, $\Delta_{c_0^+(\kappa)}^{(c)}(R)= R$ for all $R\in [0,\infty)$.
\end{proof}

\begin{corollary}
\label{c0corollary}
For any infinite cardinality $\kappa$, $\Delta_{c_0(\kappa)}^{(c)}(R)=2R$ for all $R\in [0,\infty)$.
\end{corollary}

\begin{proof}
Fix $R\in [0,\infty)$ and suppose $\Delta_{c_0(\kappa)}^{(c)}(R)<2R$.
Then there exists a point-finite cover $\mathcal{U}$ of $c_0(\kappa)$ such that $\mathcal{L}(\mathcal{U})\geq R$ and $\mathrm{diam}(\mathcal{U})<2R$.
Let $(e_\tau)_{\tau<\kappa}$ be the standard basis for $c_0(\kappa)$.
Given a finite subset $M$ of $\kappa$ and an $\varepsilon\in \{-1,1\}^M$, let 
\[A_{M,\varepsilon}=\left\{\sum_{\tau\in M}C_\tau e_\tau\ |\ \varepsilon_\tau C_\tau\in \left(\frac{1}{8}\mathrm{diam}(\mathcal{U})-\frac{1}{4}R,\frac{1}{4}\mathrm{diam}(\mathcal{U})+\frac{1}{2}R\right)\right\}.\]
Fix a finite subset $M$ of $\kappa$.
Note that for any any $\varepsilon\in \{-1,1\}^M$, $\mathrm{diam}(A_{M,\varepsilon})<R$, and so there is $U_{M,\varepsilon}\in \mathcal{U}$ such that $A_{M,\varepsilon}\subseteq U_{M,\varepsilon}$.
But $\mathrm{diam}(A_{M,\delta}\cup A_{M,\varepsilon})>\mathrm{diam}(\mathcal{U})$ whenever $\delta, \varepsilon \in \{-1,1\}^M$ are such that $\delta\neq \varepsilon$, and so in this case $U_{M,\delta}\neq U_{M,\varepsilon}$.
Thus, as $0\in A_{M,\varepsilon}$ for every every $\varepsilon\in \{-1,1\}^M$, there are at least $2^{|M|}$ different $U\in \mathcal{U}$ such that $0\in U$.
But $\kappa$ is infinite, and so has subsets of arbitrarily large finite cardinality.
That is, there are infinitely many $U\in \mathcal{U}$ such that $0\in U$, contradicting the point-finiteness of $\mathcal{U}$.
Therefore $\Delta_{c_0(\kappa)}^{(c)}(R)\geq 2R$.
Now, given $f\in c_0(\kappa)$, define $g_f\in c_0^+(\kappa)$ by $g(2\xi)=\max\{0,f(\xi)\}$ and $g(2\xi+1)=\max\{0,-f(\xi)\}$ for every $\xi<\kappa$.
The map $f\mapsto g_f$ is a 2-Lipschitz embedding, and so by Proposition \ref{embed1} and Proposition \ref{c0}, $\Delta_{c_0(\kappa)}^{(c)}(R)\leq 2R$.  That is, $\Delta_{c_0(\kappa)}^{(c)}(R)=2R$.
\end{proof}

Note that Proposition \ref{embed1}, Proposition \ref{c0}, and Corollary \ref{c0corollary} together show that the optimal distortion for a bi-Lipschitz embedding of $c_0(\kappa)$ into $c_0^+(\kappa)$ is 2.  We now come to the main result.  The proof combines techniques from both Pelant and Assouad.

\begin{theorem}
\label{lipschitz}
Let $(X,d_X)$ be a metric space with density character $\kappa$.
If there are $C\in[1,\infty)$ and $D\in[0,\infty)$ such that $\Delta_X^{(c)}(R)\leq CR+D$ for all $R\in[0,\infty)$; then for any $\lambda>0$, any $K> 2(C+\lambda)$, and any $L>\frac{(C+\lambda)D}{\lambda}$; there exists a coarse Lipschitz embedding $f\colon X\to c_0^+(\kappa)$ such that \[d_X(x,y)-L\leq \|f(x)-f(y)\|_{\infty}\leq Kd_X(x,y)\] for every $x,y\in X$.  If $D=0$, then it is possible to take $L=0$.
\end{theorem}

\begin{proof}  
Note that for any $\lambda>0$, $\Delta_X^{(c)}(R)< (C+\lambda)R$ for every $R\in (\frac{D}{\lambda},\infty)$.
Pick any $t>1$, any $0<\varepsilon<1$, any $\lambda>0$, and any point $O\in X$.
Let $K=\frac{2t(C+\lambda)}{1-\varepsilon}$.
Let $A=\left\{n\in \mathbb{Z}\ |\ t^n>\frac{D}{\lambda}\right\}$.
Then for each $n\in A$ there is a point-finite cover $\mathcal{U}_n=\{U_{n,\tau }\}_{\tau < \kappa }$ of $X$ (one can take $|\mathcal{U}_n|\leq \kappa$ by Lemma \ref{covercardinality}) such that $\mathcal{L}(\mathcal{U}_n)\geq t^n$ and $\mathrm{diam}(\mathcal{U}_n)\leq (C+\lambda)t^n$.  
For each $n\in A$ and $\tau < \kappa$, let $V_{n,\tau}=U_{n,\tau}\backslash B_{(C-1+\lambda)t^n/2}(O)$ and define $f_{n,\tau}\colon X \to \mathbb{R}^+$ by \[f_{n,\tau}(x)=K\min\left\{d_X(x,V_{n,\tau}^c), \frac{t^n}{2}\right\}\] for each $x\in X$.  
Then for each $n\in A$ and $\tau < \kappa$, $f_{n,\tau}$ is Lipschitz with $\mathrm{Lip}(f)\leq K$ and bounded by $\frac{Kt^n}{2}$.  
Note that if $f_{n,\tau}(x)>0$, then $x\in V_{n,\tau}$ and so $x\notin B_{(C-1+\lambda)t^n/2}(O)$.
Thus $f_{n,\tau}$ is supported on the complement of $B_{(C-1+\lambda)t^n/2}(O)$.  
Therefore, by the bound on $f_{n,\tau}$ and the point-finiteness of $\mathcal{U}_n$, for fixed $x\in X$ and $\eta >0$, the set \[\{(n,\tau) \in A\times \kappa\ |\ f_{n,\tau}(x)>\eta\}\] is finite.  
It follows that the map $f\colon X\to c_0^+(\kappa)$ defined by \[f(x)=\sum\limits_{(n,\tau)\in A\times \kappa} f_{n,\tau}(x)e_{n,\tau}\] for every $x\in X$ (where $\{e_{n,\tau}\}_{(n,\tau)\in A\times \kappa}$ is any enumeration of the standard basis of $c_0(\kappa)$) is a well-defined Lipschitz map with $\mathrm{Lip}(f)\leq K$.  
Now fix $x, y\in X$ such that $d_X(x,y)>(C+\lambda)\inf \{t^n\ |\ n\in A\}$ and $d_X(x,O)\geq d_X(y,O)$.
Let $n\in A$ be such that $(C+\lambda)t^n< d_X(x,y) \leq (C+\lambda)t^{n+1}$.  
Then by the triangle inequality, 
\[
d_X(x,O)  >  \frac{(C+\lambda)t^n}{2}  =  \frac{t^n}{2}+\frac{(C-1+\lambda)t^n}{2}
\]
and so $B_{t^n/2}(x)\cap B_{(C-1+\lambda)t^n/2}(O)=\emptyset$.
But $\mathcal{L}(\mathcal{U}_n)\geq t^n$, and so there is $\tau\in \kappa$ such that $B_{(1-\varepsilon)t^n/2}(x)\subseteq U_{n,\tau}$.
Therefore $f_{n,\tau}(x)\geq K\frac{(1-\varepsilon)t^n}{2}$.  Furthermore, 
\[
d_X(y,V_{n,\tau})  \geq  d_X(x,y)-\mathrm{diam}(V_{n,\tau}) 
 >  (C+\lambda)t^n-(C+\lambda)t^n 
 =  0
\]
and so $f_{n,\tau}(y)=0$.
Thus 
\begin{align*}
\|f(x)-f(y)\|  \geq  |f_{n,\tau}(x)-f_{n,\tau}(y)|
 & \geq  \frac{K (1-\varepsilon)t^n}{2} \\
 & =  \frac{K(1-\varepsilon)}{2(C+\lambda)t}(C+\lambda)t^{n+1} 
 \geq  d_X(x,y).
\end{align*}

And so, for every $x,y\in X$, \[d_X(x,y)-(C+\lambda)\inf \{t^n\ |\ n\in A\}\leq \|f(x)-f(y)\| \leq Kd_X(x,y).\]
Since $t>1$ and $0<\varepsilon<1$ were arbitrary, the theorem follows.
\end{proof}

\begin{corollary}
\label{characterize}
A metric space $(X,d_X)$ is coarse Lipschitz embeddable into $c_0(\kappa)$ for some cardinality $\kappa$ if and only if there are $C, D\in[0,\infty)$ such that $\Delta_X^{(c)}(R)\leq CR+D$ for all $R\in[0,\infty)$.  A metric space $(X,d_X)$ is bi-Lipschitz embeddable into $c_0(\kappa)$ for some cardinality $\kappa$ if and only if there is $C\in[0,\infty)$ such that $\Delta_X^{(c)}(R)\leq CR$ for all $R\in[0,\infty)$.
\end{corollary}

\begin{proof}
Suppose first that $X$ coarse Lipschitz embeds into $c_0(\kappa)$.
By Proposition \ref{c0}, $\Delta_{c_0^+(\kappa)}^{(c)}(R)\leq R$ for all $R\in[0,\infty)$ and so $\Delta_{c_0(\kappa)}^{(c)}(R)\leq 2R$ for all $R\in[0,\infty)$ by Proposition \ref{embed1} since $c_0(\kappa)$ $2$-Lipschitz embeds into $c_0^+(\kappa)$.
Thus, the implication follows from Proposition \ref{embed1}.\newline
Conversely, if there are $C,D\in[0,\infty)$ such that $\Delta_X^{(c)}(R)\leq CR+D$ for all $R\in[0,\infty)$, then $X$ coarse Lipschitz embeds (bi-Lipschitz embeds if $D=0$) into $c_0^+(\mathrm{dens}(X))$ and hence into $c_0(\mathrm{dens}(X))$ by Theorem \ref{lipschitz}.
\end{proof}

Compare the above corollary to Pelant \cite{pelant2}, who shows the uniform Stone property characterizes uniform embeddability of a metric space into $c_0(\kappa)$ for some $\kappa$; and to Baudier and Deville \cite{baudier}, who show property $\pi(\lambda)$ characterizes good-$\lambda$-embeddability of a separable metric space into $c_0$ (see \cite{baudier} for the definitions).  Lemma \ref{easy} and Corollary \ref{characterize} together show that a metric space $X$'s uniform, coarse Lipschitz, and bi-Lipschitz embeddability into $c_0^+(\kappa)$ for some $\kappa$ can all be determined from the modulus $\Delta_X^{(c)}$.
In light of this, it is natural to ask whether a metric space's coarse embeddability into $c_0(\kappa)$ for some cardinality $\kappa$ can similarly be determined from $\Delta_X^{(c)}$.
Proposition \ref{embed1} shows that the coarse Stone property is at least a necessary condition.

\begin{corollary}
\label{equivalence}
Let $X$ be a normed linear space.  The following are equivalent:
\begin{enumerate}[(i)]
\item $X$ coarsely embeds into $c_0(\kappa)$.
\item $X$ bi-Lipschitz embeds into $c_0(\kappa)$.
\item $X$ uniformly embeds into $c_0(\kappa)$.
\end{enumerate}
\end{corollary}

\begin{proof}  
(i) $\Rightarrow$ (ii):  By Lemma \ref{prelim}, $\mathrm{dens}(X)\leq \kappa$.  
By Proposition \ref{c0}, Proposition \ref{embed1}, and the fact that $c_0(\kappa)$ bi-Lipschitz embeds into $c_0^+(\kappa)$, $X$ has the coarse Stone property.
By Lemma \ref{norm}, this means there is $C\in[0,\infty)$ such that $\Delta_X^{(c)}(R)\leq CR$ for all $R\in[0,\infty)$.
And so $X$ bi-Lipschitz embeds into $c_0(\kappa)$ by Theorem \ref{lipschitz}.  \newline
(ii) $\Rightarrow$ (iii):  Clear. \newline
(iii) $\Rightarrow$ (i):   By Proposition \ref{c0}, Proposition \ref{embed2}, and the fact that $c_0(\kappa)$ bi-Lipschitz embeds into $c_0^+(\kappa)$, $X$ has the uniform Stone property.
By Lemma \ref{norm}, this means there is $C\in[0,\infty)$ such that $\Delta_X^{(c)}(R)\leq CR$ for all $R\in[0,\infty)$.
And so $X$ bi-Lipschitz embeds (and therefore coarsely embeds) into $c_0(\kappa)$ by Theorem \ref{lipschitz}.
\end{proof}

Kalton \cite{kalton1} has shown that coarse/uniform/Lipschitz embeddability into $\ell_\infty$ are also equivalent notions for normed linear spaces.  So far $\ell_\infty$ and $c_0(\kappa)$ seem to be the only spaces known to have this property, and given that $c_0(\mathfrak{c})$ (where $\mathfrak{c}$ is the cardinality of the continuum) bi-Lipschitz embeds into $\ell_\infty$ (see \cite{aharoni2}), one might ask whether the $\ell_\infty$ case actually follows from the $c_0(\kappa)$ case.  That is, can one find a bi-Lipschitz embedding of $\ell_\infty$ into $c_0(\mathfrak{c})$?  Equivalently, does $\ell_\infty$ have the coarse (or uniform) Stone property?
Pelant \cite{pelant1} and Shchepin \cite{shchepin} have shown that $\ell_\infty(\Gamma)$ fails to have the uniform Stone property when $|\Gamma|$ is large enough, but to the author's knowledge, the minimal cardinality is unknown.

\section{Spaces with the coarse Stone property}

In this section, we show directly that certain classes of metric spaces have the coarse Stone property.  In each of the examples given, $C\in[0,\infty)$ is found such that $\Delta_X^{(c)}(R)\leq CR$ for all $R\in[0,\infty)$ and so one can use Theorem \ref{lipschitz} to estimate how well $X$ bi-Lipschitzly embeds into some $c_0^+(\kappa)$.  Recall that a metric space is called \emph{locally finite} if every bounded set is finite.

\begin{proposition}
\label{locallyfintite}
If $(X,d_X)$ is a locally finite metric space, then $\Delta_X^{(c)}(R)\leq R$ for all $R\in[0,\infty)$.  Consequently, every locally finite metric space $(2+\varepsilon)$-Lipschitz embeds into $c_0^+$ for all $\varepsilon>0$.
\end{proposition}

\begin{proof}
Fix $R\in[0,\infty)$.
Let $\mathcal{U}=\{U\subseteq X\ |\ \mathrm{diam}(U)< R\}$.
Then $\mathcal{U}$ is a cover of $X$ such that $\mathcal{L}(\mathcal{U})\geq R$ and $\mathrm{diam}(\mathcal{U})\leq R$.
Now take any $x\in X$ and suppose $x\in U$.
Then $d(x,y)< R$ for all $y\in U$, and so $U\subseteq B_R(x)$.
But $|B_R(x)|<\infty$ since $X$ is locally finite.
Thus, since there are only $2^{|B_R(x)|}<\infty$ many $U\subseteq X$ such that $U\subseteq B_R(x)$, there are only finitely many $U\in \mathcal{U}$ such that $x\in U$.
This means $\mathcal{U}$ is point-finite, and thus, $\Delta_X^{(c)}(R)\leq R$.
\end{proof}

Note that Proposition \ref{locallyfintite} actually recovers the best distortion for embedding the class of locally finite metric spaces into $c_0^+$ (found by Kalton and Lancien \cite{kalton2}).  The author does not know whether the same bound holds for $\Delta_X^{(c)}$ when $X$ is an arbitrary \emph{proper} metric space (that is, a metric space whose balls are all relatively compact).

\begin{proposition}
\label{separable}
If $(X,d_X)$ is a separable metric space, then $\Delta_X^{(c)}(R)\leq 2R$ for all $R\in[0,\infty)$.  Consequently, every separable metric space $(4+\varepsilon)$-Lipschitz embeds into $c_0^+$ for all $\varepsilon>0$.
\end{proposition}

\begin{proof}  
Take any $r>0$ and any $0<\varepsilon <\frac{r}{2}$.
Let $\{x_n\}_{n=1}^\infty$ be a dense subset of $X$.
For each $n\in \mathbb{N}$, let $U_n  =  B_{\frac{r}{2}} (x_n)\backslash \bigcup\limits_{j=1}^{n-1}{B_{\varepsilon}(x_j)}$.
Then $\mathcal{U}=\{U_n\}_{n=1}^\infty$ is a cover of $X$ such that $\mathrm{diam}(\mathcal{U})\leq r$.
Now fix $x\in X$ and suppose $n\in \mathbb{N}$ is such that $d_X(x,x_n)< \varepsilon$.  
If $x\in U_j$, then $j\leq n$ by the way $\mathcal{U}$ was defined.
Therefore $x\in U_j$ for only finitely many $j\in \mathbb{N}$.
Thus, $\mathcal{U}$ is point-finite.
Now suppose $A\subseteq X$ is such that $\mathrm{diam}(A)< \frac{r}{2}-\varepsilon$.
Let $m=\min \{j\in \mathbb{N}\ | \ d_X(x_j,A)<\varepsilon\}$.
Then for each $y\in A$, $d_X(x_m,y)\leq d_X(x_m,A)+\mathrm{diam}(A)<\frac{r}{2}$ and $d_X(x_j,y)\geq \varepsilon$ for all $j<m$.
Thus $A\subseteq U_m$, and therefore $\mathcal{L}(\mathcal{U})\geq \frac{r}{2}-\varepsilon$.
Since $0<\varepsilon<\frac{r}{2}$ was arbitrary, $\Delta_X^{(u)}(r)\geq \frac{1}{2}r$ for all $r>0$.
By Lemma \ref{lineartype}, $\Delta_X^{(c)}(R)\leq 2R$ for all $R\in[0,\infty)$.
\end{proof}

Note that the 2 in Proposition \ref{separable} is optimal by Corollary \ref{c0corollary}.  At this point, it should be remarked that $\Delta_{\ell_p}^{(c)}(R)\geq \frac{(2^p+1)^{1/p}}{2}R$ for every $p\in[1,\infty)$ and $R\in[0,\infty)$.  This follows from Theorem \ref{lipschitz} and Kalton and Lancien \cite{kalton2}, who show that the best possible bi-Lipschitz embedding of $\ell_p$ into $c_0^+$ has distortion $(2^p+1)^{1/p}$.

\begin{definition}
A metric space $(T,d_T)$ is called an \emph{$\mathbb{R}$-tree} if it satisfies the following conditions:
\begin{enumerate}[(i)]
\item For any $s,t\in T$, there exists a unique isometric embedding $\phi_{s,t}\colon$ $[0,d_T(s,t)]\to T$ such that $\phi_{s,t}(0)=s$ and $\phi_{s,t}(d_T(s,t))=t$.
\item Any injective continuous mapping $\varphi \colon [0,1]\to T$ has the same range as $\phi_{\varphi(0),\varphi(1)}$.
\end{enumerate}
A \emph{rooted $\mathbb{R}$-tree} is an $\mathbb{R}$-tree $T$ paired with a point $t_0\in T$, and in this case $t_0$ is called the \emph{root} of $T$.
Given $t_1,t_2\in T$, a point $s\in T$ is said to be \emph{between} $t_1$ and $t_2$ if $s=\phi_{t_1,t_2}(x)$ for some $x\in [0,d_T(t_1,t_2)]$.
Given a nonempty subset $A$ of a rooted $\mathbb{R}$-tree $(T,t_0)$, a point $s\in T$ is called a $\emph{common ancestor}$ of $A$ if $s$ is between $t_0$ and $t$ for all $t\in A$, and is called the (necessarily unique) $\emph{last common ancestor}$ of $A$ if $s=\phi_{t_0,t} (\max \{x\in [0,d_T(t_0,t)]\ |\ \phi_{t_0,t}(x) \mbox{ is a common ancestor of } A\})$ for some $t\in A$.  One can think of an $\mathbb{R}$-tree as being a normal graph-theoretical tree with the edges ``filled in".

\end{definition}

\begin{proposition}
\label{tree}
If $(T,d_T)$ is an $\mathbb{R}$-tree (possibly non-separable), then $\Delta_T^{(c)}(R)\leq 2R$ for all $R\in[0,\infty)$.  Consequently, every $\mathbb{R}$-tree $(4+\varepsilon)$-Lipschitz embeds into $c_0^+(\kappa)$ for some $\kappa$ for all $\varepsilon>0$.
\end{proposition}

\begin{proof}
Pick any $t_0\in T$ to be the root.
Fix $R\in[0,\infty)$ and take any $n\in \mathbb{N}$.
For each $t\in T$, let 
\[U_t=\left\{s\in T\ |\ t \mbox{ is between } t_0 \mbox{ and }s \mbox{ and } d_T(t,s)\leq R+\frac{1}{n}\right\}.\]
For each $m\in \mathbb{N}\cup \{0\}$, let $\mathcal{U}_m=\left\{U_t\ |\ d_T(t_0,t)=\frac{m}{n}\right\}$.
Let $\mathcal{U}=\bigcup\limits_{m=0}^{\infty}\mathcal{U}_m$.
Then for fixed $s\in T$, there are at most $n\lceil R\rceil+1$ many $U\in \mathcal{U}$ such that $s\in U$.
Thus $\mathcal{U}$ is a point-finite cover of $T$ such that $\mathrm{diam}(\mathcal{U})\leq 2(R+\frac{1}{n})$.
Now take any $A\subseteq T$ with $\mathrm{diam}(A)<R$.
Then if $t$ is the last common ancestor of all the points in $A$ and $t'=\phi_{t_0,t}(\max\{\frac{m}{n}\ |\ m\in \mathbb{N}\cup \{0\} \mbox{ and } \frac{m}{n}\leq d_T(t_0,t)\})$, then $A\subseteq U_{t'}\in \mathcal{U}$.
This means $\mathcal{L}(\mathcal{U})\geq R$.  
Thus, since $n\in \mathbb{N}$ was arbitrary, $\Delta_T^{(c)}(R)\leq 2R$ for all $R\in[0\infty)$.
\end{proof}

\begin{proposition}
\label{finitedim}
Given $N\in \mathbb{N}$, $\Delta_{\ell_\infty^N}^{(c)}(R)= R$ for all $R\in[0,\infty)$.
\end{proposition}

\begin{proof}
Take any $n\in \mathbb{N}$.
For each $x\in \mathbb{Z}^N$, let 
\[U_x=\left\{f\in \ell_\infty^N\ |\ f(j)\in \frac{x_j}{n}+\left(-1,1+\frac{1}{n}\right)\right\}.\]
Then for fixed $f\in \ell_\infty^N$, there are at most $(2n+1)^N$ many $x\in \mathbb{Z}^N$ such that $f\in U_x$.  
Let $\mathcal{U}=\{U_x\ |\ x\in \mathbb{Z}^N\}$.
Then by the above, $\mathcal{U}$ is a point-finite cover of $\ell_\infty^N$ refined by $\{B_1(f)\}_{f\in \ell_\infty^N}$ such that $\mathrm{diam}(\mathcal{U})=2+\frac{1}{n}$.
Since every $A\subseteq \ell_\infty^N$ such that $\mathrm{diam}(A)<2$ is contained in a ball of radius $1$ (centered at $\left(\frac{\sup(\pi_j(A))-\inf(\pi_j(A))}{2}\right)_{j=1}^N$, where $\pi_j$ is the $j$-th coordinate functional), this means $\mathcal{L}(\mathcal{U})\geq 2$.  
Thus, since $n \in \mathbb{N}$ was arbitrary, $\Delta_{\ell_\infty^N}^{(c)}(2)\leq 2$.  
By Lemma \ref{C<1} and (the proof of) Lemma \ref{norm}, $\Delta_{\ell_\infty^N}^{(c)}(R)= R$ for all $R\in[0,\infty)$.
\end{proof}

\subsection*{Acknowledgements}
This research was partly supported by NSF grants DMS-1464713 and DMS-1565826.  
The author would like to thank Professors Th. Schlumprecht and F. Baudier for their encouragement and useful remarks.

\end{document}